\documentclass[english,11pt]{amsart}
\usepackage{amsmath}
\usepackage{amsthm}
\usepackage{pdfsync}
\usepackage{a4wide}
\usepackage{amssymb}
\usepackage{amsfonts}
\usepackage{color}
\usepackage{enumerate}
\usepackage{graphicx}
\usepackage{amsfonts}
\usepackage{color}
\usepackage{enumerate}
\newtheorem{Th}{Theorem}

\newtheorem{prop}{Proposition}

\newtheorem{obs}{Remark}
\newtheorem{lema}{Lemma}
\def\rr{\mathbb{R}}

\def\nn{\mathbb{N}}

\def\zz{\mathbb{Z}}
\def\ov{\overline}
\def\l{\lambda}
\def\n{\nabla}
\def\D{\Delta}
\def\q{\quad}

\def\ov{\overline}

\def\eps{\varepsilon}

\def\eps{\varepsilon}

\def\rez{\textrm{Res}}

\def\dist{\mathrm{dist}}

\def\nn{\mathbb{N}}

\def\zz{\mathbb{Z}}
\def\ov{\overline}

\def\eps{\varepsilon}

\def\dx{\ \mathrm{dx}}
\def\be{\begin{equation}}
\def\ee{\end{equation}}

\numberwithin{equation}{section} \numberwithin{Th}{section}
\numberwithin{cor}{section} \numberwithin{lema}{section}
\numberwithin{prop}{section} \numberwithin{obs}{section}
\numberwithin{Def}{section}
\title[Hardy inequalities for potentials with countable number of singularities]{Hardy inequalities for inverse square potentials with countable number of singularities}
\author[C. Cazacu]{Cristian Cazacu}
\address[C.  Cazacu]{$^1$Faculty of Mathematics and Computer Science \& The Research Institute of the University of Bucharest (ICUB),  University of Bucharest\\
	14 Academiei Street \\ 010014 Bucharest\\ Romania	\& 
	$^2$ Gheorghe Mihoc-Caius Iacob Institute of Mathematical 
		Statistics and Applied Mathematics of the Romanian Academy \\
		No.13 Calea 13 Septembrie, Sector 5 \\ 050711 Bucharest, Romania
	}
\email{cristian.cazacu@fmi.unibuc.ro}

\author[A. Marica]{Aurora Marica}
\address[A. Marica]{$^1$ Department of Mathematics and Computer Sciences, Faculty of Applied Sciences, University Politehnica of Bucharest\\ 313 Splaiul Independentei\\060042 Bucharest\\ Romania\\ 
	\& 
	$^2$ The Research Institute of the University of Bucharest (ICUB),  University of Bucharest, 90-92 Sos. Panduri, 5th District, Bucharest, Romania
}

\email{aurora.marica@upb.ro}

\begin{document}
\maketitle
\begin{abstract}
	
The Hardy Inequality (HI) for potentials with countably many singularities of the form $V=\sum_{k\in \zz}\frac{1}{|x-a_k|^2}$ is not a trivial issue. In principle, the more singular poles are, the less the Hardy constant is: it is well-known that in all the existing results about
the HI with finite number of singularities the best constants converge to 0
with the number $n$ of singularities going to infinity.  

In this note we provide an example of nontrivial HI  in right cylinders of fixed radius $R>0$ in $\rr^d$, for a potential $V$ defined above having the singularities $\{a_k\}_{k\in \zz}$ uniformly distributed on the axis of the cylinders. For this example we prove that an upper bound for the Hardy constant is $(d-2)^2/4$, the clasical Hardy constant in $\rr^d$ corresponding to one singular potential.  We also prove positive lower bounds of the Hardy constant which allow to deduce that the asymptotic behavior as $R\to 0$ of the Hardy constant coincides with $(d-2)^2/4$. The proof of the main result lies on using a nice identity due to Allegretto and
Huang \cite[Thms. 1.1, 2.1]{Allegretto} for particularly well chosen test functions.

\end{abstract}
{\bf Keywords:} Hardy inequality, singular potentials, optimal constants. asymptotic behavior. 

{\bf  MSC (2020):} 26D10, 35J75, 35B25, 46E35.

\section[Introduction]{Introduction}

The celebrated integral Hardy Inequality (abv. HI) asserts that (see, e.g. \cite{Davies, HardyNote, HardyLittlewoodPolya, Kufner, BV}) 
\begin{equation}\label{HardyIneq}
 \int\limits_\Omega |\nabla u|^2\dx\geq \frac{(d-2)^2}{4}\int\limits_\Omega\frac{u^2}{|x|^2}\dx, \qquad \forall u\in C_0^\infty(\Omega\setminus\{0\}),\end{equation}
 for any open and connected set $\Omega\subset\rr^d$, $d\geq 1$. Of course, HI \eqref{HardyIneq} is more interesting when the origin belongs to $\overline{\Omega}$, in which case the potential $1/|x|^2$ becomes unbounded.  Moreover, the constant $(d-2)^2/4$ in \eqref{HardyIneq} is optimal  when the origin belongs to $\Omega$.  Similarly, \eqref{HardyIneq} transfers to inverse-square potentials of the form $1/|x-a|^2$ with a singular pole $a\in \rr^d$  different from  the origin $x=0$.  
 
 Subsequently, various important works have analyzed HI with weighted multi-singular potentials, mainly of the form
 \begin{equation}\label{multi_potential}
 V_n:=\sum_{k=1}^{n}\frac{\mu_i}{|x-a_k|^2},
 \end{equation}
 where $\mu_k$ are constant real values and $a_k$ are given singular poles for any $k\in \{1, \ldots, n\}$. 
For a relevant bibliography we refer to 
  the papers \cite{BDE, Felli1, D, FFT, FT1, CazZua} and, more recently, to \cite{CP, CPT} and the references therein. The quoted authors obtained various lower bounds (which are the most delicate in general) and upper bounds for the best constant of the corresponding HI and applied them to the study of some critical PDEs governed by such multipolar potentials. 
  
 However, to the best of our knowledge, finding explicitly the Hardy constant is still an open issue for potentials of the form \eqref{multi_potential} for any particular domain $\Omega$, but lower and upper bounds are known. 

 In all of the above existing results about
 the HI with finite number of singularities, the best constants converge to 0
 when the number $n$ of singularities goes to infinity.  Indeed, for instance if one considers the normalized potential $V_n^1$ with $n$ singularities given by   
 $V_n^1:=\sum_{k=1}^{n}\frac{1}{|x-a_k|^2}$ and apply the optimal HI \eqref{HardyIneq} in the whole space $\rr^d$ we first obtain $$ \int_{\rr^d} |\nabla u|^2\dx \geq \frac{(d-2)^2}{4}\int_{\rr^d} \frac{u^2}{|x-a_k|^2}\dx, \quad \forall k.$$
 Summing up with respect to $k$, we reach to 
 $$ \qquad \int_{\rr^d} |\nabla u|^2\dx \geq \frac{(d-2)^2}{4n}\int_{\rr^d} V_n^1 u^2 \dx.$$ 
 Obviously, $\frac{(d-2)^2}{4n}$ tends to 0, as $n\rightarrow \infty$. 
  
This note contains one of the first results about HI for
a countable number of singularities. We are interested to obtain nontrivial estimates of the best constant in HI for a potential of the form  
\begin{equation}\label{eqq1}
V=\sum_{k\in \zz}\frac{1}{|x-a_k|^2}.
\end{equation}
 This is not a trivial issue since, by applying the above limit argument, we obtain a trivial lower bound in the HI.    

The main contribution of this note is to construct an example of a potential $V$ as in \eqref{eqq1} and a domain $\Omega$ for which the optimal Hardy constant is strictly positive.

To the best of our knowledge, the case of a countable number of
singularities has been very less investigated in the literature. The
only related result we have faced refers to the paper \cite{Felli1},
in which the authors showed that for any $\lambda< (d-2)^2/4$ and
$\{a_n\}_{n\in \nn}\subset \rr^d$ satisfying
$$\sum_{n=1}^{\infty} |a_n|^{-(d-2)}< \infty, \q \sum_{k=1}^{\infty}
|a_{n+k}-a_n|^{-(d-2)} \textrm{ is uniformly bounded  in } n, $$
$|a_n-a_m|\geq 1$ for all $n\neq m$, there exists $\delta >0$ such
that
\begin{equation}\label{bee}
\forall u\in H^{1}(\rr^d), \q  \int_{\rr^d} |\n u|^2 \dx - \lambda
\sum_{n=1}^{\infty} \int_{\rr^d} \frac{ \chi_{B_\delta(a_n)}
	u^2}{|x-a_n|^2}\dx\geq 0.
\end{equation}

Inequality \eqref{bee} represents a localization result for a potential with coutable number of singularities. However, we are interested to estimate globally the best Hardy constant. In addition, we will see later that  when  $d=3$, our sequence of singularities does not verify the above hypotheses on $\{a_n\}_{n\in \nn}$ and therefore  \eqref{bee} does not apply.

Our article is divided as follows. 
 In Section 2, we introduce a potential $V$ having the form \eqref{eqq1} for which we compute explicitly the corresponding series by using complex analysis tools. We consider $\Omega$ to be a right cylinder of a given radius $R$ and the poles $\{a_k\}_k$ to be uniformly localized on the axis of the cylinder.  We also introduce some preliminary notations and prove some auxiliary results.   In Section 3, we state the main results consisting in Theorems \ref{th1}, \ref{th2} in which we perform nontrivial upper and lower estimates for the best constant  $\lambda$ in the corresponding HI: $-\Delta \geq \lambda V$. We mainly show that the upper estimate of the best constant is the best constant for Hardy inequality on the whole space, $(d-2)^2/4$, while the lower estimate is strictly less than $(d-2)^2/4$, and tends to it when $R$ goes to 0. The proof of the main result lies on using a nice identity due to Allegretto and
Huang \cite[Thms. 1.1, 2.1]{Allegretto} for particularly well chosen test functions. 

\section{Preliminary notions and auxiliary results}

We are aimed to prove Hardy-type inequalities in a cylinder $\mathcal{C}_R$ of radius $R>0$, considering potentials with countable number of singularities equally distributed on the line $\gamma$ parametrized by $x_1=x_2=\ldots =x_d$ which represents the axis of the cylinder. The cylinder will be described mathematically in a convenient way a bit later after introducing some useful notations. 

For a given $h>0$ we consider the singular points $\{A_k\}_{k\in \zz}$ with the cartesian coordinates
$a_k=k \ov{h}$, where $\ov{h}=(h, \ldots, h)$ is a particular vector with all components equal to $h$.  This yields the inverse square potential
\begin{equation}\label{potential}
V(x)=\sum_{k\in \zz} \frac{1}{|x-k \ov{h}|^2}
\end{equation}
where $x=(x_1, \ldots, x_d)$ is a current point in $\rr^d\setminus \cup_{k\in \zz} \{A_k\}$.

Although the potential is given by a series, in fact it is explicitly computable. More precisely,

\begin{prop}\label{comp_potential}
	The potential $V$ has the explicit form
	$$V(x)= \frac{\pi}{\rho}\frac{e^{2\pi \rho}-e^{-2\pi \rho}}{e^{2\pi \rho}+e^{-2\pi \rho}-2 \cos (2\pi a)}\frac{1}{dh^2},$$
	where
	$$\rho:=\frac{\sqrt{\sum\limits_{j,k=1}^d (x_j-x_k)^2}}{dh \sqrt{2}},\quad  a:=\frac{\sum\limits_{j=1}^d x_j}{d h}.$$	
\end{prop}
\begin{proof} \textit{Step 1:} We show that
	 $$|x- k \ov{h}|^2= dh^2 ((k-a)^2 + \rho^2), \mbox{ so that }V(x)=\frac{1}{dh^2}\sum_{k\in \zz}\frac{1}{(k-a)^2 + \rho^2}.$$
	 Indeed,
	 \begin{align*}
	 |x- k \ov{h}|^2&= \sum_{j=1}^d |x_j-k h|^2=h^2 \sum_{j=1}^d \left(k^2 - 2 k \frac{x_j}{h}+ \frac{x_j^2}{h^2}\right)\\
	 & = dh^2 \Bigg(k^2 -2 k \frac{\sum\limits_{j=1}^d x_j}{dh}+\frac{\sum\limits_{j=1}^d x_j^2}{dh^2}\Bigg)\\
	 &= dh^2 \left((k-a)^2+ \frac{1}{d^2 h^2} \left(d \sum_{j=1}^d x_j^2-\left(\sum_{j=1}^d x_j\right)^2\right)\right)\\
	 &= dh^2 \left((k-a)^2+ \frac{1}{2 d^2 h^2 } \sum_{k,j=1}^{d}(x_j-x_k)^2\right)\\
	 &= dh^2 ((k-a)^2 + \rho^2).
	 \end{align*}

	 \textit{Step 2}: We now apply the following residual formula (\cite{Spiegel}, pp. 227, Proposition 7.25):
	 \begin{equation}\label{residual}\sum_{k\in \zz} f(k) =-\sum_{n=1}^\infty  \rez (\pi f(z)\cot (\pi z), z=z_n),
	 \end{equation}
with $z_n$ being the poles of $f$, to compute the series
	 $\sum_{k\in \zz} f(k)$ with
	 $f(z)=\frac{1}{(z-a)^2 + \rho^2}$. The poles of this complex function $f$ are of order one and they are given by $z_\pm=a\pm i \rho$. Therefore, in view of \eqref{residual} we have
	 \begin{align*}
	 \sum_{k\in \zz} \frac{1}{(k-a)^2 + \rho^2}&= -  \rez (\pi f(z)\cot (\pi z), z=z_-)-\rez (\pi f(z)\cot (\pi z), z=z_+)\\
	 &=: -T_--T_+.
	 \end{align*}
	 For the sake of clarity, let us completely compute the term $T_-$:
	 \begin{align*}
	 T_-&= \lim_{z\rightarrow z_-} (z-z_-)\pi f(z) \cot (\pi z) =\lim_{z\rightarrow z_-} \frac{\pi}{z-z_+} \cot (\pi z)\\
	 & =\frac{\pi}{z_--z_+} \cot (\pi z_-)= -\frac{1}{2 i \rho } \cot (\pi (a-i \rho))
	 \end{align*}
	 Similarly,
	 \begin{equation*}
	 T_+= \frac{1}{2 i \rho} \cot (\pi (a+i \rho)).
	 \end{equation*}
	 Therefore,
	 \begin{align*}
	 \sum_{k\in \zz} \frac{1}{(k-a)^2 + \rho^2} &= \frac{1}{2 i \rho } \left(\cot (\pi (a-i \rho))- \cot (\pi (a+i \rho))\right)\\
	 & =\frac{\pi}{2 \rho}\left(\frac{e^{i\pi (a-i \rho)}+ e^{-i\pi (a-i \rho)}}{e^{i\pi (a-i \rho)}-e^{-i\pi (a-i \rho)}}-\frac{e^{i\pi (a+i \rho)}+ e^{-i\pi (a+i \rho)}}{e^{i\pi (a+i \rho)}-e^{-i\pi (a+i \rho)}}\right)\\
	 & =\frac{\pi}{2 \rho} \left (\frac{e^{2\pi \rho}-e^{2\pi a i}+e^{-2\pi a i}- e^{-2\pi \rho}}{|e^{i\pi (a-i \rho)}-e^{-i\pi (a-i \rho)}|^2}- \frac{e^{-2\pi \rho}-e^{2\pi a i}+e^{-2\pi a i}- e^{2\pi \rho}}{|e^{i\pi (a+i \rho)}-e^{-i\pi (a+i \rho)}|^2}\right) \\
	 &= \frac{\pi}{\rho} \frac{ e^{2\pi \rho}-  e^{-2\pi \rho}}{e^{2\pi \rho}+ e^{-2 \pi \rho} - 2 \cos (2\pi a)}.
	 \end{align*}
	
	Combining the two steps, the proof is completed.
\end{proof}
\begin{obs} Let us also observe that $a$, $\rho$ and $|x|$ are related through the identity
	 \begin{equation}\label{relation}
	 |x|^2 = dh^2 (\rho^2 +a^2),
	 \end{equation}
	 since
	 \begin{align*}
	 2 d^2h^2\rho^2=\sum\limits_{k, j=1}^d (x_j- x_k)^2 =  2d  \sum\limits_{k, j=1}^d x_j^2 -2 \left(\sum\limits_{j=1}^d x_j\right)^2 =2d |x|^2 -2 d^2 h^2 a^2.
	 \end{align*}
\end{obs}
To simplify the computations, we consider the particular case $dh=1$. Then we obtain
$$a=\sum\limits_{j=1}^d x_j; \quad \rho =\frac{1}{\sqrt{2}} \sqrt{\sum\limits_{j, k=1}^d(x_j-x_k)^2};\quad |x|^2=\frac{\rho^2+a^2}{d}. $$

\begin{prop}The right circular cylinder $\mathcal{C}_R$ of radius $R$ with the axis $x_1=x_2=\cdots=x_d$ is characterised by $\rho\leq R\sqrt{d}$. 
\end{prop}

\textit{Proof.} The points $P(x_1^0,x_2^0,\cdots,x_d^0)$ on the lateral surface of the cylinder are characterised by $\dist(P,\gamma)=R$, where $\gamma$ is the axis of the cylinder, $\gamma:x_1=x_2=\cdots=x_d=t$. 

The plane $\pi$ passing by $P$ and perpendicular on $\gamma$ has the equation $$\pi: x_1-x_1^0+x_2-x_2^0+\cdots+x_d-x_d^0=0.$$ 

The projection of the point $P$ on the axis $\gamma$ is a point $P'$ obtained as intersection between the plane $\pi$ and the line $\gamma$, i.e. $P'(t,\cdots,t)$, where $t=\frac{x_1^0+x_2^0+\cdots+x_d^0}{d}$. 

The distance between $P$ and $\gamma$ is in fact the distance between $P$ and $P'$, i.e.
\begin{align*}R&=\dist(P,P')=\sqrt{\left(x_1^0-\frac{x_1^0+x_2^0+\cdots+x_d^0}{d}\right)^2+\cdots+\left(x_d^0-\frac{x_1^0+x_2^0+\cdots+x_d^0}{d}\right)^2}
\\&=\sqrt{|x^0|^2-\frac{(a^0)^2}{d}}=\sqrt{\frac{(\rho^0)^2}{d}}=\frac{\rho^0}{\sqrt{d}}.\end{align*}
The lateral surface of the cylinder is given by $\rho^0=R\sqrt{d}$ and the conclusion of the proposition follows for the interior of the cylinder. 

\section{Main results}\label{main_results}

The following theorems are the main results of the paper. 

\begin{Th}\label{th1} 
	 It holds 
	\begin{equation}\label{Hardy}\int_{\mathcal{C}_R} |\nabla u|^2 \dx \geq \frac{(d-2)^2}{4\pi  R\sqrt{d} \coth(\pi R \sqrt{d})} \int_{\mathcal{C}_R} V u^2 \dx, \quad \forall u \in C_0^\infty (\mathcal{C}_R\setminus\cup_i \{A_i\}).
	\end{equation}
\end{Th}

Next theorem provides lower and upper bounds for the Hardy constant $\mu(\mathcal{C}_R)$ and implicitly the asymptotic behaviour of $\mu(\mathcal{C}_R)$  as $R$ tends to zero:
\begin{Th}\label{th2} 
	We denote 
	$$\mu(\mathcal{C}_R):= \inf_{u\in C_0^\infty (\mathcal{C}_R\setminus\cup_i \{A_i\}), u\neq 0} \frac{\int_{\mathcal{C}_R} |\nabla u|^2 \dx}{\int_{\mathcal{C}_R} V u^2 \dx}.$$ 
	\begin{enumerate}[(i).]
		\item For any $R>0$ it holds
		\begin{equation}\label{bounds}
		\frac{(d-2)^2}{4\pi  R\sqrt{d} \coth(\pi R \sqrt{d})}\leq \mu(\mathcal{C}_R) \leq \frac{(d-2)^2}{4}.
		\end{equation}		
		\item Consequently, 
		\begin{equation}\label{asympt}
		\lim_{R\searrow 0} \mu(\mathcal{C}_R)=\frac{(d-2)^2}{4}. 
		\end{equation}
	\end{enumerate}
\end{Th}
\subsection*{Proofs of the main results}

Next we invoke an easy but very useful identity from \cite{Allegretto} which particularly says that
\begin{equation}\label{gen}
\int_{\mathcal{C}_R} \left( |\n u|^2 +\frac{\D \phi}{\phi}u^2\right)\dx =\int_{\mathcal{C}_R} \Big|\n u-
\frac{\n \phi}{\phi} u\Big|^2 \dx =\int_{\mathcal{C}_R} \phi^2|\n (u \phi^{-1})|^2 \dx, \q
\end{equation}
for any $u \in C_0^\infty (\mathcal{C}_R\setminus\cup_i \{A_i\})$ and  any positive function with $\phi \in C_0^\infty (\mathcal{C}_R\setminus\cup_i \{A_i\})$, admitting possible singularities at the points $A_i$.  The proof of \eqref{gen} is straightforward and requires direct integrations by parts. The difficulty comes from the election of the test functions $\phi$.

For some $\l>0$ and the singular potential $V$ in \eqref{potential},  we want to identify $\phi=\phi_\l$ in \eqref{gen} such
that
\begin{equation}\label{newh}
\int_{\mathcal{C}_R} \Big(|\n u|^2-\l V(x)u^2 \Big)\dx\geq \int_{\mathcal{C}_R} \Big( |\n u|^2
+\frac{\D \phi_\l }{\phi_\l} u^2\Big)\dx\geq 0, \q \forall u \in C_0^\infty (\mathcal{C}_R\setminus\cup_i \{A_i\}),
\end{equation}
and then, in view of \eqref{gen},  we get 
\begin{equation}\label{neww}
\int_{\mathcal{C}_R} |\n u|^2 \dx\geq \l \int_{\mathcal{C}_R} V(x)u^2 \dx, \quad  \forall u \in C_0^\infty (\mathcal{C}_R\setminus\cup_i \{A_i\}),
\end{equation}
For that it suffices 
 to find supersolutions $\phi>0$ for some $\lambda >0$  for the equation
\begin{equation}\label{supersol}
-\Delta \phi = \lambda V \phi, \quad \textrm{ in } \rr^d\setminus\cup_{i\in \zz}\{A_i\}
\end{equation}
and maximize $\lambda>0$ among the admissible pairs $(\phi, \lambda)$.
	We are looking for $\phi$ of the following form
\begin{equation}\label{phi}
\phi(x)=(e^{2\pi \rho}+e^{-2\pi \rho}-2 \cos (2\pi a))^{\alpha}:= \theta(x)^\alpha,
\end{equation}
where $\alpha\in \rr$ will be precised later in order to maximize $\lambda$.

\begin{lema}\label{lemma} We get the following expressions for $\nabla$, $\mathrm{div}$ and  $\Delta$ applied to $a$ and $\rho$:
	\begin{enumerate}[(i).]
		\item  $\nabla a=(1, \ldots, 1):= \ov{1}$; $\Delta a=0$.
		\item $\nabla \rho =\frac{d}{\rho} x- \frac{a}{\rho} \ov{1}$; $|\nabla \rho|^2=d$;   $\Delta \rho=\frac{d^2}{\rho}$; $\mathrm{div}\left(\frac{\nabla \rho}{\rho}\right)=\frac{d^2-3d}{\rho}$;
		\item $\nabla \rho \cdot \nabla a=0$;
	\end{enumerate}
\end{lema}
Lemma \ref{lemma} is proved in Section \ref{sec4}. 
In view of the basic computations in Lemma \ref{lemma} we obtain the following more detailed computations. 

\textit{Computation of $\Delta \phi$.} We successively obtain
\begin{align*}
\Delta \phi&= \alpha (\alpha-1) \theta(x)^{\alpha-2} |\nabla \theta(x)|^2 + \alpha \theta(x)^{\alpha-1}\Delta \theta(x).
\end{align*}
Since
$$|\nabla \theta(x)|^2= 4\pi^2 d(e^{2\pi \rho}-e^{-2\pi \rho})^2 + 16 \pi^2 d \sin^2(2\pi a),$$
$$\Delta \theta(x)=4\pi^2 d (e^{2\pi \rho}+ e^{-2\pi \rho}+2\cos (2\pi a))+ 2\pi (e^{2\pi \rho}-e^{-2\pi \rho})\frac{d^2-2d}{\rho}$$
we get
\begin{align*}
\Delta \phi &=\alpha \phi^{\frac{\alpha-2}{\alpha}} \Bigg\{(\alpha-1)\Big [4\pi^2 d\theta^2 + 16\pi^2 d \theta \cos(2\pi a)\Big]+\\
&+ \theta \left[2\pi (e^{2\pi\rho}-e^{-2\pi \rho})\frac{d^2-2d}{\rho}+ 4\pi^2 d  \theta + 16 \pi^2 d \cos (2\pi a)\right]
 \Bigg\}\\
 &= 2(d-2) \alpha V \phi +4\alpha^2 \pi  V \phi \rho \frac{e^{2\pi \rho}+ e^{-2\pi \rho}+2\cos (2 \pi a)}{e^{2\pi \rho}- e^{-2\pi \rho}}.
\end{align*}
Therefore, we have
\begin{equation}\label{raport}
\frac{-\Delta \phi}{V \phi} = -2 \alpha (d-2) - 4\alpha^2 \pi  \rho \frac{e^{2\pi \rho}+ e^{-2\pi \rho}+2\cos (2 \pi a)}{e^{2\pi \rho}- e^{-2\pi \rho}}.
\end{equation}

\begin{prop}\label{propC1}
	There exists a constant $C_1(R)>0$ such that
	\begin{equation}
	0\leq \rho \frac{e^{2\pi \rho}+ e^{-2\pi \rho}+2\cos (2 \pi a)}{e^{2\pi \rho}- e^{-2\pi \rho}}\leq C_1(R).
	\end{equation}
	This constant is explicitely given by $$C_1(R)= R\sqrt{d} \coth(\pi R \sqrt{d}).$$
\end{prop}

\begin{proof} Let us denote by
$$f(x):=\rho \frac{e^{2\pi \rho}+ e^{-2\pi \rho}+2\cos (2 \pi a)}{e^{2\pi \rho}- e^{-2\pi \rho}}.$$
We get
$$
f(x)\leq \rho \coth (\pi \rho)\leq R\sqrt{d} \coth(\pi R \sqrt{d})=C_1(R).$$
The last inequality holds since $g(\rho):=\rho \coth (\pi \rho)$ is an increasing function. Indeed, 
$$g'(\rho)=\coth (\pi \rho)-\frac{4\pi\rho}{(e^{\pi\rho}-e^{-\pi\rho})^2}=\frac{e^{2\pi\rho}-e^{-2\pi\rho}-4\pi\rho}{(e^{\pi\rho}-e^{-\pi\rho})^2}\geq 0.$$
\end{proof}

Using the explicit expression (\ref{raport}) of $\frac{-\Delta \phi}{V\phi}$ and the above proposition, we get
\begin{align*}
\frac{-\Delta \phi}{V \phi} &\geq  -2 \alpha (d-2) - 4\alpha^2 \pi  C_1(R)\\
&=-4 \pi  C_1(R) \left(\alpha^2 +\frac{d-2}{2 \pi  C_1(R)}\alpha\right) \\
&= - 4 \pi  C_1(R) \left(\alpha +\frac{d-2}{4 \pi  C_1(R)}\right)^2+ \frac{(d-2)^2}{4\pi  C_1(R)}.
\end{align*}
By choosing $\alpha=-\frac{d-2}{4 \pi  C_1(R)}$, we obtain 
	$$\frac{-\Delta \phi}{V \phi}\geq \frac{(d-2)^2}{4\pi  C_1(R)}.$$

In view of \eqref{supersol} and \eqref{neww} we obtain the conclusion of Theorem \ref{th1}.

\begin{obs}\label{rem1}
The constant $C_1(R)$ tends to $1/\pi$ as $R\to 0$, so that the constant in the right hand side of the Hardy inequality 
(\ref{Hardy}) has the following limit as $R\to 0$: 
$$\frac{(d-2)^2}{4\pi  R\sqrt{d} \coth(\pi R \sqrt{d})}\to \frac{(d-2)^2}{4\pi  \frac{1}{\pi}}=\frac{(d-2)^2}{4}.$$
Also, as $R\to 0$, the optimal choice of $\alpha$ becomes $\alpha=-(d-2)/4$. 

This limiting process as $R\to 0$ in the integration set $\mathcal{C}_R$ yields precisely the axis of the cylinder. Let see what happens with the function on the right hand side of identity  (\ref{raport}) in the poles of the potential which are located on the axis $x_1=\cdots=x_d=t$ at $t=k/d$. On the axis, $\rho=0$ and in the poles, $a=k$, so that as one approaches one pole, the right hand side of (\ref{raport}) tends to 
$$-2 \alpha (d-2) - 4\alpha^2 \pi  \rho \frac{e^{2\pi \rho}+ e^{-2\pi \rho}+2\cos (2 \pi a)}{e^{2\pi \rho}- e^{-2\pi \rho}}
\to -2 \alpha (d-2)-4\alpha^2.
$$
For $\alpha=-(d-2)/4$, this limits attains its maximum, which equals precisely to $(d-2)^2/4$. 	
\end{obs}

\proof[Proof of Theorem \ref{th2}] The first inequality in \eqref{bounds} is a consequence of Theorem \ref{th1}, whereas the second one is a direct consequence of \eqref{local.revers.int} in Lemma \ref{lema1.int}  below.  The asymptotic formula \eqref{asympt} is a consequence of \eqref{bounds} and Remark \ref{rem1}. 
\begin{flushright} 
	$\square$ \hfill
\end{flushright}

\begin{lema}[local results]\label{lema1.int}
	For any $\eps>0$ small enough,  there exists $U_{\eps}$ a neighborhood  of $\cup_{i=1}^{\infty} \{a_i\}$ in $\mathcal{C}_R$ such that:
	\begin{enumerate}[(i).]
		\item\label{1.int} For any $u\in C_{0}^{\infty}(U_\eps)$ it holds
		\be\label{local.hardy.int}
		\int_{\mathcal{C}_R} |\n u|^2\dx> \left(\frac{(d-2)^2}{4}-\eps \right)\int_{\mathcal{C}_R} V u^2\dx.\ee
		\item\label{2.int} There exists $u_\eps\in C_{0}^{\infty}(U_\eps)$ satisfying
		\be\label{local.revers.int}
		\int_{\mathcal{C}_R} |\n u_\eps|^2 \dx < \left(\frac{(d-2)^2}{4}+\eps\right)\int_{\mathcal{C}_R}  V u_\eps^2 \dx.
		\ee
	\end{enumerate}
\end{lema}
\begin{proof}
The proof follows similar (but in a simplified form) as some results in  \cite[Lemma 3.2 and Theorem 1.2]{Cristi_CCM},  combining the local behavior of the potential $V$ near each pole $a_i$ and the classical Hardy inequality  with an inverse square potential applied at each pole $a_i$. For the sake of clarity, let us give few details. 
	First, let us observe that for any fixed $k_0\in  \zz$, we have that 
	$$|x-a_{k}|\geq  \frac{|k-k_0|h}{2}, \quad \forall x\in B_{\frac{|k-k_0|h}{2}}(a_{k_0}), \quad \forall k \in \zz.$$
	Then, for any $x\in B_{h/2}(a_{k_0})$, we obtain 
	\begin{align*}
	\sum_{k=-\infty, k\neq k_0}^{\infty} \frac{1}{|x-a_k|^2} &\leq \sum_{k=1, k\neq k_0}^{\infty} \frac{8}{|k-k_0|^2 h^2}=\frac{8}{h^2} \sum_{k'=1}^{\infty} \frac{1}{k'^2}= \frac{4\pi^2}{3h^2}.
	\end{align*}
	Since
	\begin{align*}
	 V(x)|x-a_{k_0}|^2 &= 1 + |x-a_{k_0}|^2 \sum_{k=-\infty, k\neq k_0}^{\infty}  \frac{1}{|x-a_k|^2} 
	\end{align*}
	we get 
	\begin{equation}\label{asympt_potential}
	0< V(x)|x-a_k|^2-1\leq \frac{4\pi^2}{3h^2} |x-a_k|^2, \quad \forall k\in \zz, \quad \forall x\in B_{h/2}(a_k).   
	\end{equation} 
%
In consequence,
	\be\label{asympt.loc.int}
	\lim_{x\to a_k}V(x)|x-a_k|^2=1, \qquad k\in  \zz. 
	\ee
	For any $r>0$ small enough ($r<\min\{h/2, R\}$), it follows 
	\begin{equation}\label{cond_r}
	\cup_{k=1}^{\infty} B_r(a_k)\subset \mathcal{C}_R \quad  \textrm{ and } \quad  B_r(a_i)\cap B_r(a_j)=\emptyset,  \  \forall i\neq j.  
	\end{equation}
	In view of the classical HI, we get    \be\label{loc.bound.fall.int}
	\int_{B_r(a_k)} |\n u|^2\dx \geq \frac{(d-2)^2}{4}\int_{B_r(a_k)} \frac{u^2}{|x-a_k|^2}\dx, \quad \forall u\in C_{0}^{\infty}(B_{r}(a_k)), \quad \forall k.
	\ee
	
	Now, let $\eps>0$  and set $\delta_\eps:=4\eps/\left((d-2)^2-4\eps\right)$ so that $\delta_{\eps}\in (0, \pi^2/4)$ for $\eps$ small enough. Let us also set $r_\eps:=h\sqrt{\delta_\eps}/\pi$.  Then $r_\eps<h/2$ and, in view of \eqref{asympt_potential}, it holds 
	\be\label{limit.int}
	1 <  V(x)|x-a_k|^2 < 1+\delta_\eps, \quad \forall x\in B_{r_\eps}(a_k), \quad \forall k.
	\ee
We may assume  $\eps$ small enough such that $r_\eps$ satisfies \eqref{cond_r} and  consider then $U_\eps:=\cup_{k=1}^{\infty}  B_{r_\eps}(a_k)$.  Let be $u\in C_{0}^{\infty}(U_\eps)$ and denote  $u_k:=u_{|B_{r_\eps}(a_k)}$. Since $u$ is supported in disjoint balls that shrink around the singular poles, applying  \eqref{loc.bound.fall.int} and \eqref{limit.int} for each $u_k$, we get 
	\begin{align}\label{ineq2}
	\int_{U_\eps} |\n u|^2 \dx &= \sum_{k=1}^{\infty} \int_{B_{r_\eps}(a_k)} |\n u_k|^2 \dx \nonumber\\
	& \geq \frac{(d-2)^2}{4}\frac{1}{1+\delta_\eps} \sum_{k=1}^{\infty} \int_{B_{r_\eps}(a_k)} V u_k^2 \dx \nonumber\\
	&  =\frac{(d-2)^2}{4}\frac{1}{1+\delta_\eps} \int_{U_\eps} V u^2 \dx\nonumber\\
	&= \left(\frac{(d-2)^2}{4}-\eps\right) \int_{U_\eps} V u^2 \dx,
	\end{align}
	and the proof of \eqref{local.hardy.int} is finally obtained.
	
	\proof[Proof of \eqref{local.revers.int}] 
	
	
	Due to the optimality of $(d-2)^2/4$ in \eqref{loc.bound.fall.int} for  $k=1$ and taking  $r_{\eps}$ instead of $r$, we get that  there exists $u_{1, \eps}\in C_{0}^{\infty}(B_{r_{\eps}}(a_1))$ such that
	\be\label{ineq.invers.int}
	\int_{B_{r_{\eps}}(a_1)} |\n u_{1,\eps}|^2 \dx \leq \left(\frac{(d-2)^2}{4}+ \eps  \right) \int_{B_{r_{\eps}}(a_1)} \frac{u_{1, \eps}^2}{|x-a_1|^2} \dx.
	\ee
	Then, we consider $u_\eps:= \overline{u}_{1, \eps} $ where $\overline{u}_{1, \eps}$ is the trivial extension of $u_{1, \eps}$ to $U_\eps$.  It follows that $u_\eps$ belongs to $C_{0}^{\infty}(U_\eps)$.
	Therefore, combining \eqref{ineq.invers.int} and the left inequality in \eqref{limit.int} we get \begin{align}
	\int_{\mathcal{C}_R}|\n u_\eps|^2 \dx &= \int_{B_{r_\eps}(a_1)} |\n u_{1, \eps}|^2 \dx\nonumber\\
	& \leq \left(\frac{(d-2)^2}{4}+\eps  \right) \int_{B_{r_\eps}(a_1)} V u_{1, \eps}^{2}\dx\nonumber\\
	&= \left(\frac{(d-2)^2}{4}+\eps\right)\int_{\mathcal{C}_R} V u_\eps^2 \dx.
	\end{align}
	The proof of Lemma \ref{lema1.int} is  finished.
\end{proof}

 \begin{Th}\label{moregen}
	Assume $d\geq 3$. Then we successively obtain that
	\begin{enumerate}[(i).]
		\item \label{item1}  There exists a finite constant  $C=C(h, R)\in \rr$ which  depends on $R$ and  $h$,  so that the inequality
		\be\label{good1}
		C \int_{\mathcal{C}_R} u^2 \dx+ \int_{\mathcal{C}_R} |\n u|^2 \dx \geq \frac{(d-2)^2}{4}  \int_{\mathcal{C}_R} V u^2 \dx,
		\ee
		is verified for any $u \in C_{0}^{\infty}(\mathcal{C}_R)$.
		\item\label{item2} For any $\mu> (d-2)^2/4$  and any constant $\lambda\in \rr$ there exists $u_{\lambda, \mu}\in C_{0}^{\infty}(\mathcal{C}_R)$ such that
		\be\label{good2}
		\lambda \int_{\mathcal{C}_R} u_{\lambda, \mu}^2 \dx+ \int_{\mathcal{C}_R} |\n u_{\lambda, \mu}|^2 \dx <  \mu  \int_{\mathcal{C}_R} V u_{\lambda, \mu}^2 \dx.
		\ee
	\end{enumerate}
\end{Th}

\proof[Proof of Theorem \ref{moregen}]

We combine a cut-off argument by localizing the singularities, the standard Hardy inequality and Lemma \ref{lema1.int}.

\proof[ Proof of item \eqref{item1}] 
Let  $U:=\cup_{k=1}^{\infty}B_{h/4}(a_k)$ be a neighborhood of $\cup_{i=1}^{n}\{a_i\}$ in $\mathcal{C}_R$ constituted in a union of disjoint balls, i.e.  $B_{h/4}(a_i)\cap B_{h/4}(a_j)=\emptyset $ for any $i\neq j$. Let $u\in C_0^\infty(\mathcal{C}_R)$ and consider $\xi$ be a $C^2(\mathcal{C}_R)$ cut-off function supported in $U$, so that $0\leq \xi\leq 1$ and  $\xi \equiv 1$ in  $\cup_{k=1}^{\infty}B_{h/8}(a_k)$. Such a construction it is possible by taking for instance $\xi_k(x):=\xi_{| B_{h/4}(a_k)}(x):=g(|x-a_k|)$, where $g:[0, \infty)$ is a $C^2$ function such that $0\leq g\leq 1$, $g\equiv 1$ on $[0, h/8]$ and $g\equiv 0$ on $[h/4, \infty)$.  By integration by parts,we have
\begin{align}\label{bb1}
\int_{\mathcal{C}_R} |\n (\xi u)|^2 \dx &=\int_{\mathcal{C}_R} \xi^2  |\n u|^2 \dx -\int_{\mathcal{C}_R} (\xi \Delta \xi ) u^2 \dx \nonumber\\
&\leq \int_{\mathcal{C}_R} |\n u|^2 \dx + C \int_{\mathcal{C}_R} u^2 \dx,
\end{align}
for some constant $C>0$ depending on $g$. It is worth mentioning here that $\Delta \xi\in L^\infty(\mathcal{C}_R)$ in view of the definition of its radial profile shifted from a singular pole to the others. 
  
  Let $v_k:=\xi u_{|B_{h/4}(a_k)}$ which belongs to $C_0^\infty(B_{h/4}(a_k))$. 
  
Then, from  the standard Hardy inequality we get
\begin{align}\label{bb2}
\int_{\mathcal{C}_R} |\n (\xi u)|^2 \dx& = \sum_{k=1}^\infty \int_{B_{h/4}(a_k)} |\nabla v_k|^2 \dx \geq \frac{(d-2)^2}{4} \sum_{k=1}^\infty \int_{B_{h/4}(a_k)} \frac{v_k^2}{|x-a_k|^2} \dx \nonumber\\
& \geq \frac{(d-2)^2}{4} \sum_{k=1}^\infty \int_{B_{h/8}(a_k)} \frac{u^2}{|x-a_k|^2} \dx \nonumber\\
& =  \frac{(d-2)^2}{4} \int_{\mathcal{C}_R} V u^2 \dx -  \frac{(d-2)^2}{4} \sum_{k=1}^\infty  \int_{|x-a_k|> \frac  h 8} \frac{u^2}{|x-a_k|^2}  \dx.
\end{align}
For the second term above we successively have the estimates 
\begin{align}\label{bb3}
\sum_{k=1}^\infty  \int_{|x-a_k|> \frac h 8} \frac{u^2}{|x-a_k|^2}  \dx & :=\sum_{k=1}^{2[R/h]}\ldots + \sum_{k> 2[R/h]}\ldots \nonumber\\
& \leq \sum_{k=1}^{2[R/h]} \int_{|x-a_k|> \frac h 8} \frac{64}{h^2} u^2 \dx + \sum_{k>2[R/h]} \int_{|x-a_k|> \frac h 8} \frac{u^2}{k^2 h^2 -R^2} \dx\nonumber\\
& \leq  \frac{128}{h^2}[R/h] \int_{\mathcal{C}_R} u^2 \dx + \sum_{k> 2[R/h]} \frac{2}{k^2 h^2} \int_{\mathcal{C}_R} u^2 \dx\nonumber\\
& \leq \left(\frac{128}{h^2}[R/h] + \frac{ \pi^2 }{ 3h^2}\right) \int_{\mathcal{C}_R} u^2 \dx. 
\end{align}
Combining \eqref{bb1}-\eqref{bb3} we finally conclude the proof of item (i).

\proof[ Proof of item \eqref{item2}]   Let $\lambda\in \rr$ and $\mu> (d-2)^2/4$ be fixed. Let us also consider $\mu'$ and $\eps>0$ such that  $\mu> \mu' > (d-2)^2/4$ and  $(d-2)^2/4+\eps < \mu'$.

In view of \eqref{local.revers.int} in  Lemma \ref{lema1.int}, there exists a neighborhood $U_{\eps}:=\cup_{k=1}^{\infty}B_{r_\eps}(a_k)$ of $\cup_{k=1}^{n}\{a_k\}$ in $\mathcal{C}_R$ and $u_\eps\in C_0^\infty(U_{\eps})$ satisfying
\begin{align}\label{up.hardy}
\int_{U_{\eps}} |\n u_\eps|^2 \dx &< \left(\frac{(d-2)^2}{4}+\eps\right)\int_{U_{\eps}}  V u_\eps^2\dx\nonumber\\
& < \mu' \int_{U_{\eps}} V u_\eps^2 \dx.
\end{align}
It follows that
$$(\mu-\mu')\int_{U_{\eps}} V u_\eps^2 \dx +\int_{U_{\eps}} |\n u_\eps|^2\dx < \mu \int_{U_{\eps}} V u_\eps^2 \dx.$$
Finally, it is enough to consider $r_\eps$ small enough such that
$$\frac{\lambda}{\mu-\mu'}\leq \inf_{U_{\eps}} V,$$
to conclude that
$$\lambda \int_{U_{\eps}}u_\eps^2 \dx  +\int_{U_{\eps}} |\n u_\eps|^2\dx < \mu \int_{U_{\eps}} V u_\eps^2 \dx. $$

\section{Proof of Lemma \ref{lemma}}\label{sec4} Within this proof, we consider
$$\rho=\frac{\sqrt{\sum\limits_{j,k=1}^d(x_j-x_k)^2}}{dh\sqrt{2}}   \mbox{ and }a=\frac{\sum\limits_{j=1}^d x_j}{dh}.$$

Note that these identities can be written as
\begin{equation}\label{rhoa}2d^2h^2\rho^2=\sum\limits_{j,k=1}^d(x_j-x_k)^2 \mbox{ and } \sum\limits_{j=1}^d x_j=dha.\end{equation}

In the sum of squares $\sum\limits_{j,k=1}^d(x_j-x_k)^2$, each square $(x_m-x_n)^2$ is contained twice (once for $(j,k)=(m,n)$ and once for $(j,k)=(n,m)$). In this way, $$2d^2h^2\rho^2=2\sum\limits_{n=1}^d(x_m-x_n)^2+R_m,$$ where $R_m$ retains all squares that do not contain $x_m$. By taking partial derivative with respect to $x_m$ of this last identity, we obtain
$$4d^2h^2\rho\frac{\partial\rho}{\partial x_m}=4\sum\limits_{n=1}^d(x_m-x_n)=4dx_m-4\sum\limits_{n=1}^dx_n=4dx_m-4dha.$$
Therefore,
$$\frac{\partial\rho}{\partial x_m}=\frac{dx_m-dha}{d^2h^2\rho} \mbox{ and } \nabla\rho=\frac{dx-dha\overline{1}}{d^2h^2\rho}.$$
From this expression of $\nabla\rho$ and $x\cdot\overline{1}=dha$ obtained from (\ref{rhoa}), the norm $|\nabla\rho|$ can be easily computed as
\begin{equation}\label{normnablarho}|\nabla\rho|^2=\frac{d^2|x|^2-2d^2hax\cdot\overline{1}+d^2h^2a^2|\overline{1}|^2}{d^4h^4\rho^2}=\frac{d^2|x|^2-d^3h^2a^2}{d^4h^4\rho^2}.\end{equation}
Another identity that can be deduced from the expression of $\rho$ is
$$2d^2h^2\rho^2=\sum\limits_{j,k=1}^d(x_j-x_k)^2=\sum\limits_{j,k=1}^d(x_j^2-2x_jx_k-x_k^2)=d|x|^2-2\sum\limits_{j=1}^d x_j\sum\limits_{k=1}^d x_k+d|x|^2
=2d|x|^2-2d^2h^2a^2.$$
Consequently,

\begin{equation}d^2|x|^2-d^3h^2a^2=d^3h^2\rho^2.\label{xrhoa}\end{equation} This identity simplifies the expression of $|\nabla\rho|^2$ in (\ref{normnablarho}) as
$$|\nabla\rho|^2=\frac{d^3h^2\rho^2}{d^4h^4\rho^2}=\frac{d}{d^2h^2}.$$

NOw, let us compute the Laplacian:
\begin{eqnarray}\triangle\rho&=&\sum\limits_{m=1}^d\frac{\partial}{\partial x_m}\left(\frac{dx_m-dha}{d^2h^2\rho}\right)
=\frac{1}{d^2h^2}\sum\limits_{m=1}^d\frac{(d-1)\rho-\frac{(dx_m-dha)^2}{d^2h^2\rho}}{\rho^2}
\nonumber\\&=&\frac{1}{d^4h^4\rho^3}
\sum\limits_{m=1}^d\left[(d-1)d^2h^2\rho^2-(dx_m-dha)^2\right]\nonumber\\&=&\frac{1}{d^4h^4\rho^3}\big[(d-1)d^3h^2\rho^2-(d^2|x|^2-d^3h^2a^2)\big].\nonumber\end{eqnarray}
Using (\ref{xrhoa}), we simplify the above expression as
$$\triangle\rho=\frac{1}{d^4h^4\rho^3}\big[(d-1)d^3h^2\rho^2-d^3h^2\rho^2\big]=\frac{(d-2)d^3h^2\rho^2}{d^4h^4\rho^3}=\frac{(d-2)d}{d^2h^2\rho}.$$

Similarly,
\begin{eqnarray}\textrm{div}\Big(\frac{\nabla\rho}{\rho}\Big)&=&
\sum\limits_{m=1}^d\frac{\partial}{\partial x_m}\left(\frac{dx_m-dha}{d^2h^2\rho^2}\right)
=\sum\limits_{m=1}^d\frac{(d-1)d^2h^2\rho^2-2(dx_m-dha)^2}{d^4h^4\rho^4}
\nonumber\\&=&\frac{(d-1)d^3h^2\rho^2-2(d^2|x|^2-d^3h^2a^2)}{d^4h^4\rho^4}=
\frac{(d-1)d^3h^2\rho^2-2d^3h^2\rho^2}{d^4h^4\rho^4}\nonumber\\&=&
\frac{(d-3)d^3h^2\rho^2}{d^4h^4\rho^4}=\frac{(d-3)d}{d^2h^2\rho^2}.\nonumber\end{eqnarray}

It is easy to see that, in the general case, $\nabla a=\left(\frac{1}{dh},\cdots,\frac{1}{dh}\right)=\frac{1}{dh}\overline{1}$,
$\triangle a=0$ and
$$\nabla\rho\cdot\nabla a=\frac{dx-dha\overline{1}}{d^2h^2\rho}\cdot\frac{\overline{1}}{dh}=\frac{dx\cdot\overline{1}-dha\overline{1}\cdot\overline{1}}{d^3h^3\rho}=\frac{d^2ha-d^2ha}{d^3h^3\rho}=0.$$

The proof is completed now by taking into account that $dh=1$. 
\hfill
$\square$

{\bf Acknowledgements}. The authors were partially founded by a grant of Ministry of Research and Innovation, CNCS-UEFISCDI, project number PN-III-P1-1.1-TE-2016-2233, within PNCDI III.

\end{document}